\documentclass[12pt]{article}
\usepackage{amsmath}
\usepackage{amssymb}

\usepackage{amsthm}
\usepackage[dvipsnames]{pstricks}
\newtheorem{theorem}{Theorem}[section]
\newtheorem{lemma}[theorem]{Lemma}
\newtheorem{proposition}[theorem]{Proposition}

\newtheorem{example}[theorem]{Example}

{\theoremstyle{definition}}
{\theoremstyle{definition}}
{\theoremstyle{definition}}

\numberwithin{equation}{section}

\begin{document}

\def\C{{\mathbb C}}
\def\N{{\mathbb N}}
\def\Z{{\mathbb Z}}
\def\R{{\mathbb R}}
\def\Q{{\mathbb Q}}
\def\K{{\mathbb K}}
\def\F{{\mathcal F}}
\def\H{{\mathcal H}}
\def\E{{\mathcal E}}
\def\B{{\mathcal B}}
\def\rr{{\mathcal R}}
\def\pp{{\mathcal P}}

\def\d{ \{\!\{ }
\def\c{ \}\!\} }

\def\epsilon{\varepsilon}
\def\kappa{\varkappa}
\def\phi{\varphi}
\def\leq{\leqslant}
\def\geq{\geqslant}
\def\slim{\mathop{\hbox{$\overline{\hbox{\rm lim}}$}}\limits}
\def\ilim{\mathop{\hbox{$\underline{\hbox{\rm lim}}$}}\limits}
\def\supp{\hbox{\tt supp}\,}
\def\dim{{\rm dim}\,}
\def\ssub#1#2{#1_{{}_{{\scriptstyle #2}}}}
\def\dimk{{\ssub{\dim}{\K}\,}}
\def\ker{\hbox{\tt ker}\,}
\def\im{\hbox{\tt im}\,}
\def\spann{\hbox{\tt span}\,}
\def\supp{\hbox{\tt supp}\,}
\def\deg{\hbox{\tt \rm deg}\,}
\def\bin#1#2{\left({{#1}\atop {#2}}\right)}
\def\summ{\sum\limits}
\def\maxx{\max\limits}
\def\minn{\min\limits}
\def\limm{\lim\limits}
\def\ootimes{\,{\text{$\scriptstyle\otimes$}}\,}
\def\oo{\otimes}

\def\lll{\langle}
\def\rrr{\rangle}
\def\dd{\delta}
\def\a{\alpha}
\def\b{\beta}
\def\g{\gamma}
\def\ai{$A_{\infty}$}
\def\l{\langle}
\def\r{\rangle}
\def\lor{k\langle\langle x,y \rangle\rangle}
\def\xd#1{\frac{\partial}{\partial x_{#1}}}

\def\C{{\mathbb C}}
\def\N{{\mathbb N}}
\def\Z{{\mathbb Z}}
\def\R{{\mathbb R}}
\def\PP{\cal P}
\def\p{\rho}
\def\phi{\varphi}
\def\ee{\epsilon}
\def\ll{\lambda}
\def\a{\alpha}
\def\bb{\beta}
\def\D{\Delta}
\def\dd{\delta}
\def\g{\gamma}
\def\rk{\text{\rm rk}\,}
\def\dim{\text{\rm dim}\,}
\def\ker{\text{\rm ker}\,}
\def\square{\vrule height6pt width6pt depth 0pt}
\def\epsilon{\varepsilon}
\def\phi{\varphi}
\def\kappa{\varkappa}
\def\strl#1{\mathop{\hbox{$\,\leftarrow\,$}}\limits^{#1}}

\vskip1cm

\title{ Golod-Shafarevich-Vinberg type theorems and finiteness conditions for potential algebras}
%Any quotient of
%Universal enveloping of positive Witt algebra has a polynomial growth}
%Universal enveloping of  Witt algebra has just infinite Gelfand-Kirillov dimension}

\author{Natalia Iyudu and Stanislav Shkarin}

%\date{}

\maketitle

\begin{abstract}
We  obtain a lower estimate for the Hilbert series of Jacobi algebras and their completions by providing analogue of the Golog-Shafarevich-Vinberg theorem for potential case. We especially treat non-homogeneous situation. This estimate allows to answer number of questions arising in the work of Wemyss-Donovan-Brown on noncommutative singularities and deformation theory. In particular, we prove that the only case when a potential algebra or its completion could be finite dimensional or of linear growth, is the case of two variables and potential having terms of degree three.

\end{abstract}

% We prove that in spite the universal enveloping of positive Witt %algebra $U(W_+)$ has an intermediate growth, any its qu



\section{Introduction \label{s1}}\rm

The Golod-Shafarevich theorem provides a lower estimate for the Hilbert series of an algebra given by generators and relations in terms of degrees of generators and defining relations \cite{Golod}. Vinberg generalised the theorem for  the case when relations are not homogeneous \cite{Vinb}.

We work here in the situation when relations are not arbitrary, but obtained as (noncommutative) derivatives of one polynomial, called the potential. Such algebras are called potential or Jacobi and they frequently appear in various contexts, for example, in physics. This additional constrain allows to improve the estimate, and as a consequence, obtain important results on possible dimensions, on conditions of finite-dimensionality, other finiteness conditions, on conditions of linearity of the growth of potential algebras, which are needed in the study of contraction algebras \cite{WD, WB}. Contraction algebras serve as a noncommutative invariants of curve contractions and appear to be, as shown by Van den Bergh,  potential algebras.
%ben conj
Thus, it is important to know conditions when potential algebras are finite dimensional or of linear growth.
We  provide an answer to these questions in the  theorem below.

%\subsection{Lower estimate for $P_A$, $A$ being a potential algebra}

The methods we develop and apply in this section work for many varieties of twisted potential algebras as well. We restrict ourselves to potential algebras for the sake of clarity.

Denote by $F_k$ the $k$th graded component of the potential $F$, and by $A^{(n)}=\K\l X\r / (I + J^n)$, where $A=\K\l X\r / I$ and $J^n$ is an ideal generated by all monomials of degree strictly bigger than $n$.

We also denote here by $P_A= \sum \dim A^{(n)}t^n$ the generating function of dimensions of truncated algebras $A^{(n)}$, in stead of usual Poincare series of dimensions of components of the filtration.

We obtain a lower estimate for the Hilbert series of $A$ by estimating this series $P_A$, which is obviously
componentwise smaller than the Hilbert series of $A$:
$$H_A \geq P_A \geq (1-t)^{-1}(1-nt+nt^{k-1}-t^k)^{-1}.$$

Note that the  series $P_A$ does coincide with the Hilbert series $H_{\bar A}$ of the completion $\bar A$ of algebra $A$.

Here we give two equivalent definitions of the completion of an ideal and of an algebra (\ref{WD}).

{\bf Definition.} Consider the decreasing sequence of ideals $I^{(n)}=I+J^n$, where $J^n$ as above is an ideal generated by monomials of degree strictly bigger than $n$:
$$I^{n} \supset I^{(n+1)} \supset ...$$
Call $\bar I = \cap I^{(n)} $ {\it a completion of the ideal} $I$, and corresponding algebra $\bar A= \K \l X \r/\bar I$, {\it the completion of an algebra } $A$.

In spite of what the term of completion of an algebra should naively suggest, we have $A \supset \bar A$, however for an ideal the situation is not counterintuitive: $\bar I \supset I$.

It is easy to see that we get the same algebra, if we take quotient of the formal power series $\K \l\l X \r\r$ by the ideal generated by $I$ in $\K \l\l X \r\r$.

{\bf Definition.} $\bar A =\K \l\l X \r\r / id_{\K \l\l X \r\r} I.$

One could find these definitions for example in \cite{WD}.

%The main objective of this section is to prove the following Theorem.
%~\ref{main-growth}.

 \begin{theorem}\label{main-growth}
 Let $F\in\K^{\rm cyc}\langle x_1,\dots,x_n\rangle$ be such that $F_0={\dots}=F_{k-1}=0$, and let $n,k\in\N$ be such that $n\geq 2$, $k\geq 3$ and $(n,k)\neq(2,3)$. Then $\bar A_F$ and thereof $A_F$ are infinite dimensional.

 Furthermore, $\bar A_F$ and $A_F$ have at least cubic growth if $(n,k)=(2,4)$ or $(n,k)=(3,3)$ with cubic growth being possible in both cases, and they have exponential growth otherwise.
\end{theorem}

%% According to this theorem, the only case when  a potential algebra can be finite dimensional or of linear growth is the case of two variables and potential  having terms of degree three.

 For $n,k,m\in\N$ such that $n\geq 2$ and $m\geq k\geq 3$, denote
$$
{\cal P}_{n,k}^{(m)}=\{F\in \K^{\rm cyc}\langle x_1,\dots,x_n\rangle:F_j=0\ \text{for $j<k$ and for $j>m$}\}.
$$
Clearly, ${\cal P}_{n,k}^{(m)}$ is a vector space and ${\cal P}_{n,k}^{(k)}={\cal P}_{n,k}$. Recall that for $j\in\Z_+$ and $F\in {\cal P}_{n,k}^{(m)}$, $A_F^{(j)}$ is the quotient of $A_F$ by the ideal generated by the monomials of degree $j+1$.

Then the main lemma about zero divisors state.

\begin{lemma}\label{abc} Let $n,k\in\N$, $n\geq 2$, $m\geq k\geq 3$ and $(n,k)\neq (2,3)$, then for generic  $F\in {\cal P}_{n,k}^{(m)}$,
%. Assume also that
 $x_1a\neq 0$ in $\bar A_{F}$ for every non-zero $a\in \bar A_{F}$.
 %Then for each $j\in\Z_+$, $x_1b\neq 0$ in $A_F^{(j+1)}$ for every $b\in \K\langle x_1,\dots,x_n\rangle$ such that %$b\neq 0$ in $A_F^{(j)}$.
\end{lemma}

Using this lemma and analogue of Golod-Shafarevich-Vinberg theorem for potential algebras we give a lower estimate, which implies the statements of the main theorem above.

\begin{lemma}\label{abc1} Let $n,k\in\N$, $n\geq 2$, $m\geq k\geq 3$ and $(n,k)\neq (2,3)$, $F\in {\cal P}_{n,k}^{(m)}$ and $A=A_F$. Then $P_{A}\geq (1-t)^{-1}(1-nt+nt^{k-1}-t^k)^{-1}$. \end{lemma}

Proof given here is  a little more detailed version of the argument appeared in \cite{arxP}.

\section{Preliminary homogeneous results}

We will need two examples.
We say that potential algebra is {\it exact} if associated potential complex is exact \cite{IA}.

\begin{example}\label{potex} Let $n$ and $k$ be integers such that $k\geq n\geq 2$, $k\geq 3$ and $(n,k)\neq (2,3)$. Consider the potential $F\in{\cal P}_{n,k}$ given by
$$
F=\sum_{\sigma\in S_{n-1}} x_n^{k-n+1}x_{\sigma(1)}\dots x_{\sigma(n-1)}^\circlearrowleft,
$$
where the sum is taken over all bijections from the set $\{1,\dots,n-1\}$ to itself.
Then the potential algebra $\bar A_F=A_F$ is exact. Furthermore, $x_1u\neq 0$ for every non-zero $u\in A$.
\end{example}

%\begin{proof} We order the generators by $x_n>x_{n-1}>{\dots}>x_1$ and equip monomials %with the left-to-right degree-lexicographical ordering. The leading monomials of the %defining relations of $A$ are easily seen to be $m_n=x_n^{k-n}x_{n-1}\dots x_1$, and %$m_j=x_n^{k-n+1}x_{n-1}\dots x_{j+1}x_{j-1}\dots x_1$ with $1\leq j\leq n-1$ (after %$x_n^{k-n+1}$  we have all other $x_k$ in descending order with $x_j$ missing). There %is just one overlap of these monomials $x_n^{k-n+1}x_{n-1}\dots x_1=x_nm_n=m_1x_1$. By %Lemma~\ref{sy}, it happily resolves and the defining relations themselves form a %(finite) reduced Gr\"obner basis in the ideal of relations of $A$. This allows to %confirm that $H_A=(1-nt+nt^{k-1}-t^k)^{-1}$. Since no leading monomial of the elements %of the Gr\"obner basis starts with $x_1$, the map $u\mapsto x_1u$ from $A$ to $A$ is %injective. Hence $A$ has no non-trivial right annihilators and therefore $A$ is exact %according to Lemma~\ref{commin}.
%\end{proof}

\begin{example}\label{potex1} Let $n$ and $k$ be integers such that $n>k\geq 3$. Order the generators by $x_n>x_{n-1}>{\dots}>x_1$ and consider the left-to-right degree-lexicographical ordering on the monomials. Consider the set $M$ of degree $k-2$ monomials in $x_1,\dots,x_{n-1}$ in which each letter $x_j$ features at most once. Let $m_1,\dots,m_{n-1}$ be the top $n-1$ monomials in $M$ enumerated in such a way that $m_{n-k+1}=x_{n-1}\dots x_{n-k+2}$ $($the biggest one$)$. Now define the potential  $F\in{\cal P}_{n,k}$  by
$$
F=x_nx_{n-1}\dots x_{n-k+1}^\circlearrowleft+\sum_{1\leq j\leq n-1\atop j\neq n-k+1} x_jx_nm_j^\circlearrowleft.
$$
Then the potential algebra $A=A_F$ is exact.  Furthermore, $x_1u\neq 0$ for every non-zero $u\in A$.
\end{example}

%\begin{proof} We use the same order as above. The leading monomials of the defining %relations of $A$ are easily seen to be $x_{n-1}\dots x_{n-k+1}$ and $x_nm_j$ for $1\leq %j\leq n-1$. Again, there is just one overlap of these monomials
%$x_nx_{n-1}\dots x_{n-k+1}=x_n(x_{n-1}\dots x_{n-k+1})=(x_nx_{n-1}\dots %x_{n-k+2})x_{n-k+1}$, which happens to resolve according to Lemma~\ref{sy}. The rest of %the proof is the same as for the previous example.
%\end{proof}

\begin{lemma}\label{geni} Let $\K$ be uncountable field, $n,k\in\N$, $n\geq 2$, $k\geq 3$ and $(n,k)\neq (2,3)$. Then for a generic $F\in{\cal P}_{n,k}$, $x_1a\neq 0$ in $A_F$ for every non-zero $a\in A_F$.
\end{lemma}

\begin{proof} Let $F_0$ be the potential provided by the appropriate (depending on whether $k\geq n$ or $k<n$) Example~\ref{potex} or Example~\ref{potex1}. Then $x_1a\neq 0$ in $A_{F_0}$ for every non-zero $a\in A_{F_0}$ and $H_{A_{F_0}}=(1-nt+nt^{k-1}-t^k)^{-1}$.
%Lemma~\ref{commin1} guarantees
As was noticed by Ufnarovskij, the generic Hilbert series is minimal, hence $H_{A_F}=(1-nt+nt^{k-1}-t^k)^{-1}$ for generic $F\in{\cal P}_{n,k}$. Applying %Lemma~\ref{maxra}
Lemma~3.9 from \cite{IA}
%[Iyudu,Smoktunowicz]
to the map $a\mapsto x_1a$ from $A_F$ to $A_F$, we now see that $\dim x_1(A_F)_j\geq \dim x_1(A_{F_0})_j$ for all $j$ for generic $F\in{\cal P}_{n,k}$. Since $\dim x_1(A_{F_0})_j=\dim (A_{F_0})_j=\dim (A_{F})_j$ for generic $F$, the map $a\mapsto x_1a$ from $A_F$ to itself is injective for generic $F$.
\end{proof}

\section{Main Lemma about zero divisors}

\begin{lemma}\label{abc} Let $n,k\in\N$, $n\geq 2$, $m\geq k\geq 3$ and $(n,k)\neq (2,3)$, then for generic  $F\in {\cal P}_{n,k}^{(m)}$,
%. Assume also that
 $x_1a\neq 0$ in $\bar A_{F}$ for every non-zero $a\in \bar A_{F}$.
 %Then for each $j\in\Z_+$, $x_1b\neq 0$ in $A_F^{(j+1)}$ for every $b\in \K\langle x_1,\dots,x_n\rangle$ such that %$b\neq 0$ in $A_F^{(j)}$.
\end{lemma}

\begin{proof} Assume the contrary. Then there exist $j\in\Z_+$ and $a\in \K\langle x_1,\dots,x_n\rangle$ such that $a\neq 0$ in $A_F^{(j)}$ and $x_1a=0$ in $A_F^{(j+1)}$. The latter means that
\begin{equation*}%\label{tata1}
x_1a=\sum_{j\in N} u_jr_{s(j)}v_j\,({\rm mod}\,\,J^{(j+1)}),
\end{equation*}
where $r_j=\delta_{x_j}F$, $N$ is a finite set, $s$ is a map from $N$ to $\{1,\dots n\}$, $u_j,v_j$ are non-zero homogeneous elements of $\K\langle x_1,\dots,x_n\rangle$ such that the degree of each $u_jv_j$ does not exceed $j-k+2$, and the equality $f=g\,({\rm mod}\,\,J)$ means $f-g\in J$. Let $m$ be the lowest degree of $u_jv_j$ and $N'=\{j\in N:\deg u_jv_j=m\}$. Then the smallest degree part of the above display reads
$$
x_1a_{m+k-2}=\sum_{j\in N'} u_j\rho_{s(j)}v_j\ \ \text{in $\K\langle x_1,\dots,x_n\rangle$,}
$$
where $\rho_j=\delta_{x_j}F_k$. Note that automatically $a_q=0$ for $q<m+k-2$.

Consider two cases, depending on whether  the  left hand side of the lower degree term is zero or not.

{\it Case I.} If $x_1 a_{m+k-2}=0$, then we have an equality
$$ 0=\sum u_j\rho_{s(j)}v_j,$$
which means that it is a syzygy of  $\rho_i$th.
We will use the following fact.

 {\bf Definition.} We say that $F\in \K^{\rm cyc}\langle x_1,\dots,x_n\rangle$ is {\it S-trivial} if the module of syzygies of $A_F$ presented by generators $x_1,\dots,x_n$ and relations $r_1,\dots,r_n$ with $r_j=\delta_{x_j}F$ is generated by trivial syzygies and the syzygy $\sum [x_j,{r}_j]$.
 %provided by Lemma~\ref{gen1}.
%\end{proof}[]

\begin{proposition} Any  homogeneous potential in general position  is $S$-trivial.
\end{proposition}

\begin{proof}
 First, we observe that
 %if $G\in {\cal P}_{n,m}$ and $H_{A_G}=(1-nt+nt^{m-1}-t^m)^{-1}$, then $G$ is $S$-trivial. Indeed, otherwise
 any 'extra' syzygy will 'drop' the dimension of the corresponding component of the ideal of relations thus increasing the dimension of the component of the algebra compared to the minimal Hilbert series.
 This follows from the fact that the module of
  syzygies of relations
  is generated by polynomials obtained from the  resolutions of ambiguities of the Gr\"obner basis \cite{Gr}.

 Thus, if Hilbert series is minimal, all syzygies must be generated by one syzygy.
 Now we remind the fact observed, for example, by Ufnarovskij \cite{ufn}, that generic series is minimal.
\end{proof}

Due to the above proposition we have
$$
\sum_{j\in N} u_j\rho_{s(j)}v_j\,=\sum_k \alpha_k\sum_i[\rho_i,x_i]\, \beta_k.$$

 We can rewrite the initial equality
  \begin{equation*}%\label{tata1}
x_1a=\sum_{j\in N} u_jr_{s(j)}v_j\,({\rm mod}\,\,J^{(j+1)}),
\end{equation*}
 adding the appropriate combination of syzygies (which is zero) to the right hand side:
 $$x_1 a =\sum u_j r_{s(j)} v_j -\sum \alpha_k \sum_i [r_i,x_i] \beta_k \,\,(mod J^{(j+1)}).$$

 Now the lowest term of the right hand side become bigger. Indeed, the lowest terms cancel, since
 the lowest degree term of
$ \sum_{j\in N} u_j r_{s(j)}v_j$ coincides with  the lowest degree term of
$\sum_{j\in N} u_j\rho_{s(j)}v_j$, and
the lowest degree term of  $\sum \alpha_k[r_i,x_i] \beta_k$    coincides with
 the lowest degree term  of $\sum \alpha_k[\rho_i,x_i] \beta_k$.

Now for newly rewritten equality we again write down the lowest terms of the left and right hand side.
If for the left hand side it is again zero: $x_1 a_{m'+k-2}=0,$ we continue the process of lifting up the lowest degree on the right hand side, as described in Case I.

%The condition imposed upon $A_{F_k}$ means that the ideal $K$ generated by $\rho_1,\dots,\rho_n$ satisfies $x_1b\in K\,\Longrightarrow b\in K$. Hence, by the above display,

{\it Case II}.
Otherwise we are in Case II, when the left hand side
$x_1 a_{m'+k-2}\neq 0$. The lowest term is from the left hand side now.  We will find a presentation, where the lower degree of the left hand side is bigger.

 Due to the fact that the statement of our lemma~\ref{abc} holds true in the case of homogeneous potential, which is proved in Lemma~\ref{geni}, we can deduce that
$$
a_{m'+k-2}=\sum_{p\in M} f_p\rho_{t(p)}g_p,
$$
where $M$ is a finite set, $t$ is a map from $M$ to $\{1,\dots n\}$, $f_p,g_p$ are non-zero homogeneous elements of $\K\langle x_1,\dots,x_n\rangle$, such that the degree of each $f_pg_p$ is $m-1$. Now we replace $a$ by
$$
a'=a-\sum_{p\in M} f_pr_{t(p)}g_p.
$$
Note that $a=a'$ in $\bar A_F$ and therefore $a=a'$ in $A_F^{(j)}$ and $x_1a=x_1a'$ in $A_F^{(j+1)}$. So $a'$ satisfies the same properties as $a$ with the only essential difference being that $a'_{m'+k-2}=0$.

Now we can repeat the process chipping off the homogeneous degree-components of left and right hand sides of equality from bottom up one by one until at the final step we arrive to a contradiction with $a\neq 0$ in $A_F^{(j)}$.
\end{proof}

\section{Proof of the main estimate using Golod-Shafarevich-Vinberg type argument}

%Let us introduce some necessary  denotations.
%According to Lemma~\ref{abc}, we then have that for each $j\in\Z_+$, $x_1b\neq 0$ in $B^{(j+1)}$ for every $b\in %\K\langle x_1,\dots,x_n\rangle$ such that $b\neq 0$ in $B^{(j)}$.
%This property allows us to pick inductively (starting with $M_0=\{1\}$) sets $M_j$ of monomials of degree $j$ such %that $M_{j+1}\supseteq x_1M_j$ and $N_j=M_0\cup{\dots}\cup M_j$ is a linear basis in $B^{(j)}$ for each $j\in\Z_+$. %For every $j$, let $B^+_j$ be the linear span of $N_j$ and $B^{++}_j$ be the linear span of $N_j\setminus %x_1N_{j-1}$ in $\K\langle x_1,\dots,x_n\rangle$.

\begin{lemma}\label{abc1} Let $n,k\in\N$, $n\geq 2$, $m\geq k\geq 3$ and $(n,k)\neq (2,3)$, $F\in {\cal P}_{n,k}^{(m)}$ and $A=\bar A_F$. Then $P_{A}\geq (1-t)^{-1}(1-nt+nt^{k-1}-t^k)^{-1}$. \end{lemma}

\begin{proof} First, observe that exchanging the ground field $\K$ for a field extension does not affect the series $P_A$. Thus we can without loss of generality assume that $\K$ is uncountable. For $j\in\Z_+$, let $b_j$ be Taylor coefficients of the rational function $Q(t)=(1-t)^{-1}(1-nt+nt^{k-1}-t^k)^{-1}$ (that is, $Q(t)=\sum b_jt^j$) and $a_j=\min\{\dim A_G^{(j)}:G\in {\cal P}_{n,k}^{(m)}\}$. The proof will be complete if we show that $a_j=b_j$ for all $j\in\Z_+$. Denote $P=\sum a_jt^j$. First, note that Examples~\ref{potex} and~\ref{potex1}, provide $G\in {\cal P}_{n,k}\subseteq {\cal P}_{n,k}^{(m)}$ for which $H_G=(1-nt+nt^{k-1}-t^k)^{-1}$. It immediately follows that $P_G=Q$. By definition of $P$ (minimality of $a_j$), we then have $P\leq Q$, that is, $a_j\leq b_j$ for all $j\in\Z_+$.

%By Lemmas~\ref{miser1} and~\ref{geni},
 For a generic $G\in {\cal P}_{n,k}^{(m)}$ $P_{A_G}=P$ according to nonhomogeneous generalisation of Ufnarovskij's observation. Moreover, by lemma~\ref{abc}, for a generic $G\in {\cal P}_{n,k}^{(m)}$, we have that
   $x_1a\neq 0$ in $\bar A_G$ for every non-zero $a\in \bar A_{G}$. In particular, we can pick a single $G\in {\cal P}_{n,k}^{(m)}$ such that for $B=\bar A_G$, $P_B=P$ and $x_1a\neq 0$ in $\bar A_{G}$ for every non-zero $a\in \bar A_{G}$.

    According to Lemma~\ref{abc}, we then have that for each $j\in\Z_+$, $x_1b\neq 0$ in $B^{(j+1)}$ for every $b\in \K\langle x_1,\dots,x_n\rangle$ such that $b\neq 0$ in $B^{(j)}$.
This property allows us to pick inductively (starting with $M_0=\{1\}$) sets $M_j$ of monomials of degree $j$ such that $M_{j+1}\supseteq x_1M_j$ and $N_j=M_0\cup{\dots}\cup M_j$ is a linear basis in $B^{(j)}$ for each $j\in\Z_+$. For every $j$, let $B^+_j$ be the linear span of $N_j$ and $B^{++}_j$ be the linear span of $N_j\setminus x_1N_{j-1}$ in $\K\langle x_1,\dots,x_n\rangle$. Clearly $P_B=\sum (\dim B^+_j)t^j$ and therefore $a_j=\dim B^+_j$ for all $j\in\Z_+$.

Let also $\pi^{(j)}$ be the natural projection of $\K\langle x_1,\dots,x_n\rangle$ onto the linear span of monomials of length $\leq j$ along $J^{(j)}$ (in fact, $\pi^{(j)}: \K[X]\to\K[X]^{(j)}$). As usual, let $V$ be the linear span of $x_1,\dots,x_n$, $r_j=\delta_{x_j}G$, $R$ be the linear span of $r_1,\dots,r_n$ and $I$ be the ideal generated by $r_1,\dots,r_n$ (=the ideal of relations of $B$). For the sake of brevity denote $\Phi=\K\langle x_1,\dots,x_n\rangle$.

Now we argue in a way similar to the Golod-Shafarevich-Vinberg theorem \ref{Golod, Vinb}, but incorporating at some point the syzygy $\sum [x_j,r_j]=0$ which holds for any potential algebra.

Obviously,
$I=VI+R\Phi$. Then $\pi^{(j+1)}(I)=V\pi^{(j)}(I)+\pi^{(j+1)}(R\Phi)$ for every $j\in\Z_+$. Using the definition of $B_j^+$ and the fact that each $r_j$ starts at degree $\geq k-1$, we obtain
$$
\pi^{(j+1)}(I)=V\pi^{(j)}(I)+\pi^{(j+1)}(RB^+_{j+2-k}).
$$
Since $\sum[x_j,r_j]=0$ in $\Phi$, we can get rid of $r_1x_1$:
$$
V\pi^{(j)}(I)+\pi^{(j+1)}(RB^+_{j+2-k})=V\pi^{(j)}(I)+\pi^{(j+1)}(R'B^+_{j+2-k}+r_1B^{++}_{j+2-k})
$$
where $R'$ is the linear span of $r_2,\dots,r_n$. Thus
$$
\pi^{(j+1)}(I)=V\pi^{(j)}(I)+\pi^{(j+1)}(R'B^+_{j+2-k}+r_1B^{++}_{j+2-k}).
$$
Hence
$$
\begin{array}{l}
\dim \pi^{(j+1)}(I)\leq \dim V\pi^{(j)}(I)+\dim R'B^+_{j+2-k}+\dim r_1B^{++}_{j+2-k}
\\ \qquad\qquad =n\dim \pi^{(j)}(I)+(n-1)\dim B^+_{j+2-k}+\dim B^{++}_{j+2-k}.
\end{array}
$$
Plugging the equalities $\dim B^+_j=a_j$, $\dim B^{++}_j=a_j-a_{j-1}$  (since multiplication by $x_1$ is injective, we assume also $a_{s}=0$ for $s<0$), and $\dim \pi^{(j)}(I)=1+n+{\dots}+n^j-a_j$ into the inequality in the above display, we get
$$
1+{\dots}+n^{j+1}-a_{j+1}\leq n+{\dots}+n^{j+1}-na_j+na_{j+2-k}-a_{j+1-k}.
$$
Hence $a_{j+1}\geq na_j-na_{j+2-k}+a_{j+1-k}-1$ for $j\in\Z_+$. On the other hand, it is easy to see that the Taylor coefficients $b_j$ of $Q$ satisfy $b_{j+1}=nb_j-nb_{j+2-k}+b_{j+1-k}-1$ for $j\geq k-1$. It is also elementary to verify that $a_j=b_j$ for $0\leq j\leq k-1$. Now for $c_j=b_j-a_j$, we have $c_j=0$ for $0\leq j\leq k-1$, $c_j\geq 0$ for $j\geq k$ and $c_{j+1}\leq nc_j-nc_{j+2-k}+c_{j+1-k}$ for $j\geq k-1$. The only sequence satisfying these conditions is easily seen to be the zero sequence. Hence $a_j=b_j$ for all $j\in\Z_+$, which completes the proof.
\end{proof}

Now Theorem~\ref{main-growth} is a direct consequence of Lemma~\ref{abc1}. Indeed, every potential $F$ on $n$ variables starting in degree $\geq k$ belongs to ${\cal P}_{n,k}^{(m)}$ for $m$ large enough and Lemma~\ref{abc1}  provides at least cubic growth of $\bar A_F$ in the case $(n,k)=(3,3)$ or $(n,k)=(2,4)$ and exponential growth otherwise.

\vfill\break

\small\rm

\end{document}